\documentclass[a4paper,12pt]{amsart}
\usepackage{euscript,
mathrsfs,color,latexsym,marginnote}
\usepackage{graphicx}
\usepackage{amsmath,amssymb,amsthm,amsfonts}
\usepackage{amssymb}
\usepackage[active]{srcltx}
\usepackage{geometry}
\usepackage{hyperref}
\usepackage{color}
\newtheorem{thm}{Theorem}[section]

\newtheorem{lem}[thm]{Lemma}
\newtheorem{prop}[thm]{Proposition}
\newtheorem{rem}[thm]{Remark}

\newcommand{\R}{\mathbb R}

\newcommand{\N}{\mathbb N}

\begin{document}
	
	\title[Local behaviour of the second order derivatives ]{Local behaviour of the second order derivatives of  solutions to $p$-Laplace equations}
	
	\author{Felice Iandoli, Domenico Vuono}

	\email[Felice Iandoli]{felice.iandoli@unical.it}
	\email[Domenico Vuono]{domenico.vuono@unical.it}
	\address{Dipartimento di Matematica e Informatica, Università della Calabria,
		Ponte Pietro Bucci 31B, 87036 Arcavacata di Rende, Cosenza, Italy}
	

	

	\begin{abstract}
		We consider the equation $- \Delta_{p} u = f(x)$ in $\Omega,$ where $\Delta_{p}$ is the $p$-Laplace operator. We provide $L^{\infty}$-type estimates for the second derivatives of  solutions  when $p$ approaches to $2$.
	\end{abstract}

 	\maketitle
  
	\section{Introduction}
	
	This paper is concerned  with the regularity of the second order derivatives  of weak solutions to:
	\begin{equation} \label{eq:problema}
		- \Delta_{p} u = f(x)  \quad \text{in } \Omega\,
	\end{equation}
	where $\Delta_{p} u := - \operatorname{div} (|\nabla u|^{p-2} \nabla u)$, for $p>1$, is the $p$-Laplace operator and $\Omega $ is a domain of $\R^{n}$ with $n \ge 2$. 
	For $u\in W^{1,p}_{loc}(\Omega)$, the weak formulation of \eqref{eq:problema} is the following
	\begin{equation}\label{soluzionedebole}
		\int_\Omega |\nabla u|^{p-2}(\nabla u,\nabla \psi)\,dx=\int_\Omega f \psi \,dx \quad \forall \, \psi\in C^{\infty}_c(\Omega).
	\end{equation}
    
	Denote by $B_R(x_0)$ the open ball of radius $R$ centered at $x_0$.
	The main result is the following.

    \begin{thm}\label{thm:main}
	    Let $\Omega$ be an open set of $\R^n$, $n\geq 2$, $x_0\in \Omega$ and $B_{2R}(x_0)\subset\subset \Omega$. Let $u$ be a weak solution of~\eqref{eq:problema}. Assume that $f\in  W^{2,l}_{loc}(\Omega)$, with $l>n/2$. For $k>0$ fixed, there exists  $\mathfrak{C}:=\mathfrak{C}(k,l,n)>0$ small enough such that if $|p-2|<\mathfrak{C}$, then  
     $$\||\nabla u|^k D^2u\|_{L^\infty(B_R(x_0))}\le \mathcal{C},$$
     where $\mathcal{C}=\mathcal{C}(k,n,l,R,p,\| \nabla u\|_{L^{\infty}(B_{2R}(x_o))},\|f\|_{W^{2,l}( B_{2R}(x_o))})$ is a positive constant.

\end{thm}

\begin{rem}
The following example shows that the presence of the weight $|\nabla u|^k$ is necessary on order to establish Theorem \ref{thm:main}.
Consider  
$u(x_1,\ldots,x_n):=|x_1|^{p'}/p'$,  $p'=\frac{p-1}{p}$. This function satisfies $\Delta_p u=1$ in $\mathbb{R}^n$, for any $n\geq 1$. 
One can note that $\partial_{x_1x_1}^2 u=C_p|x_1|^{\frac{2-p}{p-1}},$ which of course is not in $L^{\infty}(\Omega)$, for any $\Omega$   open set such that $\Omega\cap \{x_1=0\}\neq \emptyset$ and for any $p>2$.
On the other hand 
$|\nabla u|^kD^2 u$ is locally in  $L^{\infty}$ if $p\le 2+k$. Note that in such an example, when $k$ goes to zero, we obtain $p\leq 2$. We do not know the exact relationship between $p$ and $k$ in general, as it depends on the behavior of the constant in the Calderón-Zygmund inequality (see Remark \ref{rmk:cappuccio}), which is not explicit.

\end{rem}

\begin{rem}\label{rmk:cappuccio} The range of $p$ for which our theorem holds true depends upon the Calderón-Zygmund constant, more precisely  it is sufficient that
     \begin{equation*}
         2-\frac{1}{C(n,\hat q)}< p <\min \left\{2+\frac{1}{\hat q-1} ,2+\frac{1}{C(n,\hat q)}\right\},
     \end{equation*}
     with $C(n,\hat q)$ given by \eqref{eq:CZ},  and  
     \begin{equation}\label{cappuccio1}\hat{q}:=\hat{q}(k,l,n)=\max\left\{\left(\frac1k+1\right)\frac{2\nu}{\nu-2},\frac{2-k}{k}\frac{l\nu}{l(\nu-2)-\nu}\right\},\end{equation}
     where $\frac{2l}{l-1}<\nu<2^*$ if $n\geq3$ and  $\frac{2l}{l-1}<\nu<r$ for some fixed $r>\frac{2l}{l-1}$ if $n=2$.
\end{rem}

\begin{rem}\label{holdero}
    Note that $f \in W^{2,l}_{loc}(\Omega)$, with $l>n/2$, implies $f\in C^{0,\beta'}_{loc}(\Omega)$, for some $\beta'\in(0,1)$, see \cite[Theorem 12.55]{leoni}.
   \end{rem}

\noindent Let us now discuss the state of the art of the regularity results concerning the equation \eqref{eq:problema}. 
\noindent In general, solutions of \eqref{eq:problema} are not classical. It is well known (see \cite{DB,GT,LIB,T}) that for every $p>1$ and under suitable assumptions on the source term,  solutions of \eqref{eq:problema} are of class $C^{1,\beta}_{loc}(\Omega)\cap C^2(\Omega\setminus Z_u)$, where $Z_u$ denotes the set where the gradient vanishes.

Second order estimates for 
$p$-Laplace equations have been extensively investigated in the past decades, leading to significant advancements in the understanding of regularity properties. 

\noindent In \cite{Cma}, under the minimal assumption $f\in L^2$, the $W^{1,2}$ regularity of the stress field $|\nabla u|^{p-2}\nabla u$ was established. Moreover, for $1<p<2$ the authors established that 
$u\in W^{2,2}$ under stronger assumptions on the source term. Similar regularity results for the stress field were obtained in \cite{antonuovo,Lou}. 

\noindent For $1<p<3$, the $W^{2,2}$ regularity of solutions was further developed in \cite{DS,MMS} for $f\in W^{1,1}\cap L^q$, with $q>n$. Additionally, in these works, by imposing a sign condition on the source term, it was shown that for $p\ge 3$, the solutions  belong to $W^{2,q}$, with $q<(p-1)/(p-2)$ (this is an optimal result). In particular, in \cite{MMS}, the $W^{1,2}$ regularity was investigated for terms of the form $|\nabla u|^{\alpha-1}\nabla u$. More general results can be found in \cite{cellina, HS, LU, S1, S2}.

\noindent The anisotropic counterpart of these results, which involves significant technical challenges, has been explored in \cite{Anto1, Anto2, BMV, CRS}. Additionally, we highlight recent advances in regularity theory for the vectorial case, presented in \cite{BaCiDiMa, Cmavett, Minchievic, MMSV}.

\noindent For $p$ close to 2, more general regularity estimates, in the spirit of a natural Calderón–Zygmund theory, were obtained in \cite{beni, MRS}. In particular, in \cite{beni}, new estimates on the third derivatives of the solution have been recently introduced. We also mention an $L^{\infty}$ estimate for second derivatives in some very specific cases, such as 
$p$-harmonic functions in the planar setting $n=2$ \cite{IM}. In this framework, a general theory has only been developed in \cite{AKM,2,4,245,FFM,DM,DM2,DM3,Dong,12,13,KuuMin,16}. In addition to their intrinsic interest, these estimates play a key role in the analysis of qualitative and quantitative properties of p-Laplace equations, including classification results and stability, see for instance \cite{CCR}.

\noindent In this paper, exploiting Serrin's technique \cite{serrin} (see also \cite{DEG,MO}), we establish
$L^{\infty}$-type estimates for the second derivatives of the solution. To the best of our knowledge, this is the first result in this direction. 
    
\textbf{Strategy of the proof.}
In the sequel, we will use the notation
	$u_i := u_{x_i}$ , for $i = 1, . . . , n$, to indicate the partial derivative of $u$ with respect to $x_i$.  The second derivatives of $u$ will be
	denoted with $u_{ij}:=(\partial ^2 u)/(\partial ^2 x_ix_j),$ $i,j = 1,...,n$. We denote  by $\nabla u$ and by $D^2u$ respectively the gradient vector and the Hessian matrix of $u$.

We explain the proof when $n\geq 3$, being the case $n=2$ similar.	Let $x_0\in \Omega$ and  $B_{2R}:=B_{2R}(x_0)\subset\subset\Omega$. First of all we consider the regularized problem
	\begin{equation} \label{eq:problregol}
		\begin{cases}
			-\operatorname{div}\left( (\varepsilon +|\nabla u_{\varepsilon}|^{2})^{\frac{p-2}{2}} \, \nabla u_{\varepsilon} \right) = f(x) & \text{in } B_{2R}(x_0) \\
			u_{\varepsilon} = u & \text{on } \partial B_{2R}(x_0),
		\end{cases}
	\end{equation}
	where $\varepsilon \in (0,1)$. We assume for the moment $f\in C^2(B_{2R})$, this hypothesis will be removed by means of a density argument.
    
    The existence of  a weak solution of \eqref{eq:problregol} follows by a classical minimization procedure. Notice that by standard regularity results \cite{DiKaSc,GT},  the solution $ u_\varepsilon$ is regular.
	In \cite{DiKaSc,DB,L2} (see also \cite[Section 5]{Anto2}) it has been proven that for any compact set $K\subset \subset B_{2R}(x_0)$ the sequence ${u}_\varepsilon$ is uniformly bounded in $C^{1,\beta '}(K)$, for some $0<\beta'<1$. Moreover, it follows that
		${u}_\varepsilon\rightarrow  u$ in the norm $\|\cdot\|_{C^{1,\beta}(K)},$ for all $0\le \beta<\beta '$.

        With all these ingredients (see the last proof of the paper) it is enough to prove  that the function $g_{\varepsilon}(x):=(\varepsilon+|\nabla u_\varepsilon|^2)^{\frac{k}{2}}|u_{\varepsilon,i j}|$ is locally uniformly bounded in $\varepsilon$ and then pass to the limit for $\varepsilon\rightarrow 0$. This estimate is the core of the paper and it is obtained in Proposition \ref{prop:moser}. In order to get such an estimate we use a Moser's iteration process \cite{MO}. The starting point for this iteration process is to prove that 
	\begin{align}\label{ancora}
    \begin{split}
    \|g_{\varepsilon}\|_{L^{2^*q}(B_{h'})}^q\leq \mathcal{C}\frac{q}{h-h'}\|g_{\varepsilon}\|_{L^{\nu(q-\hat{s})}(B_h)}^{q-\hat{s}}, \quad \hat{s}:=\frac1k,
    \end{split}
\end{align}
for any $q\geq 1/k$ and any $0<h'<h<R$, and 
where $\mathcal{C}$ is independent of $\varepsilon$, $B_h$ and $B_{h'}$ are the balls centered at $x_0$ of radii respectively $h$ and $h'$,
 and $\nu<2^*$. With such an estimate one can start the Moser's iteration: with the choice in \eqref{definizionediq_n} we build the sequences $p_{i}\rightarrow\infty$ and $h_i\rightarrow R'$, for any $R'<R$, for $i\rightarrow\infty$ in such a way that, by applying $i$-times \eqref{ancora}, we obtain 
$ \|g_{\varepsilon}\|_{L^{p_i}(B_{h_i})}\leq \mathcal{C},$ for a constant $\mathcal{C}$ independent of $i$ and $\varepsilon$. Passing to the limit in $i$ we obtain the desired bound on the $L^{\infty}$-norm of $g_{\varepsilon}$ on $B_{R'}$. \\
The inequality \eqref{ancora} is crucial for this paper and it needs several ingredients. In
 the  preparation Lemma \ref{Lemmalinearizzatodue}, we study the second linearized of the equation \eqref{soluzionedebole} and, by choosing a suitable test function, we obtain the $L^2$-estimates for $\nabla g_{\varepsilon}^q$, which, by means of Sobolev embedding, is bounded from below by the l.h.s. of \eqref{ancora}. The estimate in \eqref{stimalinearizzatosecondo} in Lemma \ref{Lemmalinearizzatodue} depends on some powers of the second derivatives of the solution $u_{\varepsilon}$. To handle these powers of $u_{\varepsilon,ij}$ we make a repeated use, among  other things, of the result in \cite{beni} which provides uniform estimates for the $L^q$ norms of $u_{\varepsilon,ij}$ (see Remark \ref{regolarizziamo}). This is the content of the Prop. \ref{dinodino}.  After that, we conclude the proof of  Theorem \ref{thm:main} passing to the limit in $\varepsilon$ and removing the $C^2$ assumption on the forcing term $f$.

\textbf{Organization.} The paper is organized in two sections. In the first one we recall some  results that are used in the proof. In the second section we give the proof of Theorem \ref{thm:main}.
	
	\section{Notations and Preliminary results}\label{notations}
We say that $\Omega$ is a domain if it is open and connected.
We recall  the Calder\'on-Zygmund inequality  (see e.g. \cite[Corollary 9.10]{GT}):
	
	\begin{lem}\label{Zygmund}
    Let $\Omega$ be a domain of $\R^n$, $n\geq 2$, and let $w \in W^{2,q}_{0}(\Omega)$. Then there exists a positive constant $C(n,q)$ such that 
		\begin{equation} \label{eq:CZ}
			\|D^{2}w\|_{L^{q}(\Omega)} \le C(n,q) \|\Delta w\|_{L^{q}(\Omega)}.
		\end{equation}
	\end{lem}

    The following theorem is very important for our purposes. 

\begin{thm}[\cite{beni,MRS}]\label{teoremacalderon}
		Let $\Omega$ be a domain of $\R^{n}$, $n\geq 2$, and $u\in C^{1,\beta}_{loc}(\Omega)$ be a weak solution to \eqref{eq:problema}, with $f(x)\in W^{1,1}_{loc}(\Omega)\cap C_{loc}^{0,\beta'}(\Omega).$ Let $q \ge 2$ and $p$ be such that $$2-\frac{1}{C(n,q)}<p < \min\left\{2+\frac{1}{q-1},2+\frac{1}{C(n,q)}\right\},$$ with $C(n,q)$ given by \eqref{eq:CZ}. Then we have $u \in W^{2,q}_{loc}(\Omega).$
		
\end{thm}
	We shall make repeated use  of the following remark.
\begin{rem}\label{regolarizziamo}
    Let $u,q$ and $p$ as in Theorem \ref{teoremacalderon} and let $u_\varepsilon$ be a solution of \eqref{eq:problregol}. Recall Remark \ref{holdero} and the fact that  the source term is H\"older continuous.
    In order to establish our results we shall also exploit the fact that for any ball $B_{2R}(x_0) \subset\subset \Omega$ there exists positive constant $\mathcal{C}$ depending on $n$, $q$, $R$, $p$, $ \| \nabla u\|_{L^{\infty}(B_{2R}(x_0))}$,$\|f\|_{C^{0,\beta'}(B_{2R}(x_0))},\|f\|_{W^{1,1}(B_{2R}(x_0))}$ such that 
\begin{equation}\label{stimaderivatasecondaregolarizzate}
	\|D^{2}u_{\varepsilon}\|_{L^{q}(B_{R}(x_0))}\le \mathcal{C},
		\end{equation}
   This is a consequence of Theorem \ref{teoremacalderon}, see \cite[Proof of Theorem 1.3]{beni}.
\end{rem}

	\section{Proof of the main result}\label{risultatiprincipali} 
We introduce, for $i,j=1,\ldots, n$, the notation 
\begin{equation}\label{g}
g_{\varepsilon}:=(\varepsilon+|\nabla u_\varepsilon|^2)^{\frac{k}{2}}|u_{\varepsilon,i j}|.
\end{equation}
We begin with a preliminary lemma in which we  estimate from above the $L^2$-norm of the gradient of the function $g_{\varepsilon}^{q}$ localized in some fixed ball $B_{2R}$. In this lemma we shall study the second linearized of the weak formulation of the regularized problem \eqref{eq:problregol}, i.e.
		\begin{equation}\label{equazione debole3}
			\int_{B_{2R}} (\varepsilon+|\nabla u_\varepsilon|^2)^{\frac{p-2}{2}}(\nabla u_\varepsilon,\nabla \varphi)\,dx=\int_{B_{2R}} f \varphi \,dx \quad \forall \varphi\in C^{\infty}_c(B_{2R}).
		\end{equation}

\begin{lem}\label{Lemmalinearizzatodue}
Let $\Omega$ be a domain in $\R^n$, $n\ge 2$, and $B_{2R}\subset\subset \Omega$. Let $u_\varepsilon$ be a weak solution of the regularized problem \eqref{eq:problregol}, with $p>1$. We assume that $f\in C^2(\Omega)$. Then, for any  $k \in \R$, $q\ge 1$  we have

\begin{align}\label{stimalinearizzatosecondo}
    \begin{split}
    \|\nabla g_\varepsilon^q\psi\|_{L^2(B_{2R})}^2
       &\le   q^2C(p,k)\int_{B_{2R}} (\varepsilon+|\nabla u_\varepsilon|^2)^{kq-1} | u_{\varepsilon,ij}|^{2q-2}|D^2u_\varepsilon|^4 \psi^2\,dx 
        \\& + q C(p) \int_{B_{2R}} (\varepsilon+|\nabla u_\varepsilon|^2)^{kq} | u_{\varepsilon,ij}|^{2q}|\nabla \psi|^2\,dx \\&+ qC(p)\int_{B_{2R}}|f_{ij}|(\varepsilon+|\nabla u_\varepsilon|^2)^{\frac{2kq+2-p}{2}}|u_{\varepsilon,ij}|^{2q-1}\psi^2\,dx,
    \end{split}
\end{align}
for any nonnegative smooth function $\psi$ compactly supported in $B_{2R}$ and where $C(p,k)$ and $C(p)$ are positive constants.

\end{lem}
\begin{proof}
   Let $u_{\varepsilon}$ be as in the statement. Since $f\in C^{2}(\Omega)$, by standard regularity results (see \cite{DB,GT}), we deduce that $u_\varepsilon\in C^3(B_{2R})$. Since 
    \begin{equation}\label{gradio}
\begin{aligned}
    \nabla\left((\varepsilon+|\nabla u_{\varepsilon}|^2)^{\frac{qk}{2}}|u_{\varepsilon,ij}|^{q}\right)
    &=q(\varepsilon+|\nabla u_{\varepsilon}|^2)^{\frac{qk}{2}}|u_{\varepsilon,ij}|^{q-2}u_{\varepsilon,ij}\nabla u_{\varepsilon,ij}\\
    &+qk(\varepsilon+|\nabla u_{\varepsilon}|^2)^{\frac{qk-2}{2}}|u_{\varepsilon,ij}|^{q}D^2u_{\varepsilon}\nabla u_{\varepsilon},
    \end{aligned}\end{equation}
    by means of the Young inequality it is sufficient to estimate from above the $L^2$-norms of the two summands on the r.h.s. above. Let us start with the first summand in the r.h.s. of \eqref{gradio}.
     We consider the second linearized equation of the regularized problem.
		We fix $i,j=1,...,n$ and we use $\varphi_{ij}:=(\partial ^2 \varphi)/(\partial ^2 x_ix_j)\in C^{\infty}_c(B_{2R})$ in \eqref{equazione debole3}, thanks to the regularity of $u_{\varepsilon}$ and $g_{\varepsilon}$, integrating twice by parts, we obtain
		
		\begin{align}
	 \label{eq:lin secregolarizzato}
  \begin{split}
			& \int_{B_{2R}} (\varepsilon+|\nabla u_\varepsilon|^2)^{\frac{p-2}{2}} (\nabla u_{\varepsilon,ij},\nabla \varphi)\,dx \\
		+& (p-2) \int_{B_{2R}} (\varepsilon+|\nabla u_\varepsilon|^2)^{\frac{p-4}{2}} (\nabla u_\varepsilon,\nabla u_{\varepsilon,j})(\nabla u_{\varepsilon,i}, \nabla \varphi)\,dx \\
			+& (p-2)(p-4) \int_{B_{2R}} (\varepsilon+|\nabla u_\varepsilon|^2)^{\frac{p-6}{2}} (\nabla u_\varepsilon,\nabla u_{\varepsilon,j})(\nabla u_{\varepsilon,i},\nabla u_\varepsilon)(\nabla u_\varepsilon,\nabla \varphi)\,dx \\
			+& (p-2) \int_{B_{2R}} (\varepsilon+|\nabla u_\varepsilon|^2)^{\frac{p-4}{2}} (\nabla u_{\varepsilon,ij},\nabla u_\varepsilon)(\nabla u_\varepsilon,\nabla \varphi)\,dx\\
			+& (p-2) \int_{B_{2R}} (\varepsilon+|\nabla u_\varepsilon|^2)^{\frac{p-4}{2}} (\nabla u_{\varepsilon,i},\nabla u_{\varepsilon,j})(\nabla u_\varepsilon,\nabla \varphi) \,dx\\
			+& (p-2) \int_{B_{2R}} (\varepsilon+|\nabla u_\varepsilon|^2)^{\frac{p-4}{2}} (\nabla u_{\varepsilon,i},\nabla u_\varepsilon)(\nabla u_{\varepsilon,j},\nabla \varphi)\,dx = \int_{B_{2R}}f_{ij}\varphi\,dx.\\
   \end{split}
		\end{align}

\

\
Consider  $\psi$ as in the statement. 
For fixed numbers $k\in \R$ and $q\ge 1$, we consider the following test function 
\begin{equation}\label{test1}
    \varphi (x):= (\varepsilon+|\nabla u_\varepsilon|^2)^{\frac{2kq+2-p}{2}}|u_{\varepsilon,ij}|^{2q-2}u_{\varepsilon,ij}\psi^2\in C_c^1(B_{2R}). 
\end{equation}

So, we get 
\begin{equation*}
\begin{split}
 &\nabla \varphi= (2q-1)(\varepsilon+|\nabla u_\varepsilon|^2)^{\frac{2kq+2-p}{2}}|u_{\varepsilon,ij}|^{2q-2}\nabla u_{\varepsilon,ij}\psi^2 
 \\ &+(2kq+2-p)(\varepsilon+|\nabla u_\varepsilon|^2)^{\frac{2kq-p}{2}}|u_{\varepsilon,ij}|^{2q-2}u_{\varepsilon,ij}\psi^2 D^2u_\varepsilon\nabla u_\varepsilon 
\\&+2(\varepsilon+|\nabla u_\varepsilon|^2)^{\frac{2kq+2-p}{2}}|u_{\varepsilon,ij}|^{2q-2}u_{\varepsilon,ij}\psi \nabla \psi:=J_1+J_2+J_3.
\end{split}
\end{equation*}

Plugging the function  $\varphi$ defined in \eqref{test1} in the equation \eqref{eq:lin secregolarizzato}, we obtain 

	\begin{align}
	 \label{eq:lin secregolarizzato2}
  \begin{split}
			0=& \int_{B_{2R}} (\varepsilon+|\nabla u_\varepsilon|^2)^{\frac{p-2}{2}} (\nabla u_{\varepsilon,ij},J_1+J_2+J_3) \,dx
   \\+& (p-2) \int_{B_{2R}} (\varepsilon+|\nabla u_\varepsilon|^2)^{\frac{p-4}{2}} (\nabla u_\varepsilon,\nabla u_{\varepsilon,j})(\nabla u_{\varepsilon,i},J_1+J_2+J_3) \,dx \\
			+& (p-2)(p-4) \int_{B_{2R}} (\varepsilon+|\nabla u_\varepsilon|^2)^{\frac{p-6}{2}} (\nabla u_\varepsilon,\nabla u_{\varepsilon,j})(\nabla u_{\varepsilon,i},\nabla u_\varepsilon)\times 
   \\& \quad\qquad\quad\qquad\quad\quad\times (\nabla u_\varepsilon,J_1+J_2+J_3)\,dx \\
			+& (p-2) \int_{B_{2R}} (\varepsilon+|\nabla u_\varepsilon|^2)^{\frac{p-4}{2}} (\nabla u_{\varepsilon,ij},\nabla u_\varepsilon)(\nabla u_\varepsilon,J_1+J_2+J_3)\,dx\\
			+& (p-2) \int_{B_{2R}} (\varepsilon+|\nabla u_\varepsilon|^2)^{\frac{p-4}{2}} (\nabla u_{\varepsilon,i},\nabla u_{\varepsilon,j})(\nabla u_\varepsilon,J_1+J_2+J_3)\,dx \\
			+& (p-2) \int_{B_{2R}} (\varepsilon+|\nabla u_\varepsilon|^2)^{\frac{p-4}{2}} (\nabla u_{\varepsilon,i},\nabla u_\varepsilon)(\nabla u_{\varepsilon,j},J_1+J_2+J_3)\,dx
            \\&- \int_{B_{2R}}f_{ij}(\varepsilon+|\nabla u_\varepsilon|^2)^{\frac{2kq+2-p}{2}}|u_{\varepsilon,ij}|^{2q-2}u_{\varepsilon,ij}\psi^2\,dx
            \\&=:I_1+\cdot\cdot\cdot +I_{18}-I_{19}.\\
   \end{split}
\end{align}

Let us estimate the term $I_1+I_{10}$.
For $p\ge 2$, we get 

\begin{align}\label{stima p piu grande di due}
    \begin{split}
        I_1+I_{10}&= (2q-1)\int_{B_{2R}} (\varepsilon+|\nabla u_\varepsilon|^2)^{kq} |u_{\varepsilon,ij}|^{2q-2}\psi ^2|\nabla u_{\varepsilon,ij}|^2\,dx\\&+ (2q-1)(p-2) \int_{B_{2R}} (\varepsilon+|\nabla u_\varepsilon|^2)^{kq-1} |u_{\varepsilon,ij}|^{2q-2}\psi ^2(\nabla u_{\varepsilon,ij},\nabla u_\varepsilon)^2\,dx
        \\& \ge (2q-1)\int_{B_{2R}} (\varepsilon+|\nabla u_\varepsilon|^2)^{kq} |u_{\varepsilon,ij}|^{2q-2}\psi ^2|\nabla u_{\varepsilon,ij}|^2\,dx.
    \end{split}
\end{align}

In the case $p<2$, using Cauchy-Schwarz
inequality we obtain 

\begin{equation}\label{stima p piu grande di due2}
    I_1+I_{10}\ge (2q-1)(p-1)\int_{B_{2R}} (\varepsilon+|\nabla u_\varepsilon|^2)^{kq} |u_{\varepsilon,ij}|^{2q-2}\psi ^2|\nabla u_{\varepsilon,ij}|^2\,dx.
\end{equation}

Using \eqref{stima p piu grande di due} and \eqref{stima p piu grande di due2} in \eqref{eq:lin secregolarizzato2} we have 

\begin{align}\label{primastima}
    \begin{split}
        &(2q-1)\min\{1,p-1\}\int_{B_{2R}} (\varepsilon+|\nabla u_\varepsilon|^2)^{kq} |u_{\varepsilon,ij}|^{2q-2}\psi ^2|\nabla u_{\varepsilon,ij}|^2\,dx\\& \le |I_2|+\cdot\cdot\cdot +|I_9|+|I_{11}|+\cdot\cdot\cdot +|I_{18}|+|I_{19}|.
    \end{split}
\end{align}

Now we estimate the right-hand side of \eqref{primastima}. First, we estimate the terms $|I_2|+|I_4|+|I_7|+|I_{11}|+|I_{13}|+|I_{16}|$.
By Cauchy-Schwarz inequality we obtain 

\begin{align}\label{stima1}
    \begin{split}
       &|I_2|+|I_4|+|I_7|+|I_{11}|+|I_{13}|+|I_{16}|=\\& |2kq+2-p| \Big|\int_{B_{2R}} (\varepsilon+|\nabla u_\varepsilon|^2)^{kq-1} |u_{\varepsilon,ij}|^{2q-2}u_{\varepsilon,ij}\psi ^2(\nabla u_{\varepsilon,ij},D^2u_\varepsilon\nabla u_\varepsilon)\,dx\Big| 
       \\&+ (2q-1)|p-2| \Big|\int_{B_{2R}} (\varepsilon+|\nabla u_\varepsilon|^2)^{kq-1} |u_{\varepsilon,ij}|^{2q-2}\psi ^2(\nabla u_\varepsilon,\nabla u_{\varepsilon,j})(\nabla u_{\varepsilon,ij},\nabla u_{\varepsilon,i})\,dx\Big|
      \\&+  (2q-1)|p-2||p-4| \Big|\int_{B_{2R}} (\varepsilon+|\nabla u_\varepsilon|^2)^{\frac{2kq-4}{2}} |u_{\varepsilon,ij}|^{2q-2}\psi ^2\times
      \\&  \qquad\qquad\qquad\qquad\qquad\qquad\qquad\times(\nabla u_{\varepsilon,j},\nabla u_\varepsilon)(\nabla u_{\varepsilon,i},\nabla u_\varepsilon)(\nabla u_{\varepsilon,ij},\nabla u_\varepsilon)\,dx \Big|
      \\&+ |2kq+2-p||p-2| \Big|\int_{B_{2R}} (\varepsilon+|\nabla u_\varepsilon|^2)^{\frac{2kq-4}{2}} |u_{\varepsilon,ij}|^{2q-2}u_{\varepsilon,ij}\psi ^2\times\\&\qquad\qquad\qquad\qquad\qquad\qquad\qquad\qquad\times(\nabla u_{\varepsilon,ij},\nabla u_\varepsilon)(\nabla u_\varepsilon,D^2u_\varepsilon\nabla u_\varepsilon)\,dx\Big|
      \\&+ (2q-1)|p-2| \Big|\int_{B_{2R}} (\varepsilon+|\nabla u_\varepsilon|^2)^{kq-1} |u_{\varepsilon,ij}|^{2q-2}\psi ^2(\nabla u_{\varepsilon,i},\nabla u_{\varepsilon,j})(\nabla u_{\varepsilon,ij},\nabla u_\varepsilon)\,dx\Big|
      \\&+ (2q-1)|p-2| \Big|\int_{B_{2R}} (\varepsilon+|\nabla u_\varepsilon|^2)^{kq-1} |u_{\varepsilon,ij}|^{2q-2}\psi ^2(\nabla u_{\varepsilon,i},\nabla u_{\varepsilon})(\nabla u_{\varepsilon,ij},\nabla u_{\varepsilon,j})\,dx\Big|
      \\& \le  qC(p,k) \int_{B_{2R}} (\varepsilon+|\nabla u_\varepsilon|^2)^{\frac{2kq-1}{2}} |u_{\varepsilon,ij}|^{2q-2}\psi ^2|\nabla u_{\varepsilon,ij}||D^2u_\varepsilon|^2\,dx,
    \end{split}
\end{align}

where $C(p,k)$ is a positive constant.\\ 

We estimate the terms $|I_3|+|I_{12}|$. Using Cauchy-Schwarz inequality we have 
\begin{equation}\label{stima2}
\begin{split}
    &|I_3|+|I_{12}| = 2 \Big|\int_{B_{2R}} (\varepsilon+|\nabla u_\varepsilon|^2)^{kq}|u_{\varepsilon,ij}|^{2q-2}u_{\varepsilon,ij}\psi (\nabla u_{\varepsilon,ij},\nabla \psi) \,dx\Big|
    \\& +2 |p-2|\Big|\int_{B_{2R}} (\varepsilon+|\nabla u_\varepsilon|^2)^{kq-1}|u_{\varepsilon,ij}|^{2q-2}u_{\varepsilon,ij}\psi  (\nabla u_{\varepsilon,ij},\nabla u_\varepsilon)(\nabla u_\varepsilon,\nabla \psi) \,dx\Big|
    \\& \le  C(p) \int_{B_{2R}} (\varepsilon+|\nabla u_\varepsilon|^2)^{kq}|u_{\varepsilon,ij}|^{2q-1}\psi  |\nabla u_{\varepsilon,ij}||\nabla \psi|\,dx,
    \end{split}
\end{equation}
where $C(p)$ is a positive constant.
In a similar way we obtain

\begin{align}\label{stima3}
    \begin{split}
      &|I_5|+|I_8|+|I_{14}|+|I_{17}|=
      \\&+ |2kq+2-p||p-2| \Big|\int_{B_{2R}} (\varepsilon+|\nabla u_\varepsilon|^2)^{\frac{2kq-4}{2}} |u_{\varepsilon,ij}|^{2q-2}u_{\varepsilon,ij}\psi ^2\times\\&\qquad\qquad\qquad\qquad\qquad\qquad\qquad\qquad\times(\nabla u_{\varepsilon,j},\nabla u_\varepsilon)(\nabla u_{\varepsilon,i},D^2u_\varepsilon\nabla u_\varepsilon)\,dx\Big|
      \\&+ |2kq+2-p||p-2||p-4| \Big| \int_{B_{2R}} (\varepsilon+|\nabla u_\varepsilon|^2)^{\frac{2kq-6}{2}} |u_{\varepsilon,ij}|^{2q-2} \times \\& \qquad\qquad\qquad\qquad\qquad\qquad  \times u_{\varepsilon,ij}\psi ^2(\nabla u_{\varepsilon,j},\nabla u_\varepsilon)(\nabla u_{\varepsilon,i},\nabla u_\varepsilon)(\nabla u_\varepsilon,D^2u_\varepsilon\nabla u_\varepsilon)\,dx\Big|
       \\&+ |2kq+2-p||p-2|\Big| \int_{B_{2R}} (\varepsilon+|\nabla u_\varepsilon|^2)^{\frac{2kq-4}{2}} |u_{\varepsilon,ij}|^{2q-2}u_{\varepsilon,ij}\psi ^2\times\\&\qquad\qquad\qquad\qquad\qquad\qquad\qquad\qquad\times(\nabla u_{\varepsilon,j},\nabla u_{\varepsilon,i})(\nabla u_\varepsilon,D^2u_\varepsilon\nabla u_\varepsilon)\,dx\Big|
        \\&+ |2kq+2-p||p-2| \Big|\int_{B_{2R}} (\varepsilon+|\nabla u_\varepsilon|^2)^{\frac{2kq-4}{2}} |u_{\varepsilon,ij}|^{2q-2}u_{\varepsilon,ij}\psi ^2\times\\&\qquad\qquad\qquad\qquad\qquad\qquad\qquad\qquad\times(\nabla u_{\varepsilon},\nabla u_{\varepsilon,i})(\nabla u_{\varepsilon,j},D^2u_\varepsilon\nabla u_\varepsilon)\,dx\Big|
         \\&\le q C(k, p)  \int_{B_{2R}} (\varepsilon+|\nabla u_\varepsilon|^2)^{kq-1} |u_{\varepsilon,ij}|^{2q-1}\psi ^2 |D^2u_\varepsilon|^3\,dx,
       \end{split}
\end{align}
where $C(p,k)$ is a positive constant.

For the  terms $|I_6|+|I_9|+|I_{15}|+|I_{18}|$, we have the following estimate
\begin{align}\label{stima4}
    \begin{split}
       &|I_6|+|I_9|+|I_{15}|+|I_{18}|
       \\& =2 |p-2|\Big| \int_{B_{2R}} (\varepsilon+|\nabla u_\varepsilon|^2)^{kq-1} |u_{\varepsilon,ij}|^{2q-2}u_{\varepsilon,ij}\psi (\nabla u_{\varepsilon,j},\nabla u_{\varepsilon})(\nabla u_{\varepsilon,i},\nabla \psi)\,dx\Big|
       \\&+2 |p-2||p-4|  \Big|\int_{B_{2R}} (\varepsilon+|\nabla u_\varepsilon|^2)^{\frac{2kq-4}{2}} |u_{\varepsilon,ij}|^{2q-2}u_{\varepsilon,ij}\psi \times\\&\qquad\qquad\qquad\qquad\qquad\times(\nabla u_{\varepsilon,j},\nabla u_\varepsilon) (\nabla u_{\varepsilon,i},\nabla u_{\varepsilon})(\nabla u_{\varepsilon},\nabla \psi)\,dx\Big|
       \\& +2 |p-2| \Big|\int_{B_{2R}} (\varepsilon+|\nabla u_\varepsilon|^2)^{kq-1} |u_{\varepsilon,ij}|^{2q-2}u_{\varepsilon,ij}\psi (\nabla u_{\varepsilon,i},\nabla u_{\varepsilon,j})(\nabla u_{\varepsilon},\nabla \psi)\,dx\Big|
       \\&+2 |p-2| \Big|\int_{B_{2R}} (\varepsilon+|\nabla u_\varepsilon|^2)^{kq-1} |u_{\varepsilon,ij}|^{2q-2}u_{\varepsilon,ij}\psi (\nabla u_{\varepsilon,i},\nabla u_{\varepsilon})(\nabla u_{\varepsilon,j},\nabla \psi)\,dx\Big|
       \\ &\le  C(p) \int_{B_{2R}} (\varepsilon+|\nabla u_\varepsilon|^2)^{\frac{2kq-1}{2}} |u_{\varepsilon,ij}|^{2q-1}|D^2u_\varepsilon|^2 |\nabla \psi| |\psi|\,dx,
    \end{split}
\end{align}

where $C(p)$ is a positive constant.

For the last term $|I_{19}|$, obviously, we have 
\begin{equation}\label{terminef_ij}
   |I_{19}|\le   \int_{B_{2R}}|f_{ij}|(\varepsilon+|\nabla u_\varepsilon|^2)^{\frac{2kq+2-p}{2}}|u_{\varepsilon,ij}|^{2q-1}\psi^2\,dx.
\end{equation}

Using \eqref{stima1}, \eqref{stima2}, \eqref{stima3}, \eqref{stima4} and \eqref{terminef_ij} in the estimate \eqref{primastima} we obtain

\begin{align}\label{stimasulsecondolinearizzato}
    \begin{split}
    &(2q-1)\min\{1,p-1\}\int_{B_{2R}} (\varepsilon+|\nabla u_\varepsilon|^2)^{kq} |u_{\varepsilon,ij}|^{2q-2}\psi ^2|\nabla u_{\varepsilon,ij}|^2\,dx
       \\&\le  qC(p,k) \int_{B_{2R}} (\varepsilon+|\nabla u_\varepsilon|^2)^{\frac{2kq-1}{2}} |u_{\varepsilon,ij}|^{2q-2}\psi ^2|\nabla u_{\varepsilon,ij}||D^2u_\varepsilon|^2\,dx
       \\&+  C(p) \int_{B_{2R}} (\varepsilon+|\nabla u_\varepsilon|^2)^{kq}|u_{\varepsilon,ij}|^{2q-1}\psi  |\nabla u_{\varepsilon,ij}||\nabla \psi|\,dx
       \\&+ q C( p,k)  \int_{B_{2R}} (\varepsilon+|\nabla u_\varepsilon|^2)^{kq-1} |u_{\varepsilon,ij}|^{2q-1}\psi ^2 |D^2u_\varepsilon|^3\,dx
        \\&+ C(p) \int_{B_{2R}} (\varepsilon+|\nabla u_\varepsilon|^2)^{\frac{2kq-1}{2}} |u_{\varepsilon,ij}|^{2q-1}|D^2u_\varepsilon|^2 |\nabla \psi| \psi\,dx
        \\&+ \int_{B_{2R}}|f_{ij}|(\varepsilon+|\nabla u_\varepsilon|^2)^{\frac{2kq+2-p}{2}}|u_{\varepsilon,ij}|^{2q-1}\psi^2\,dx
        \\&:=\mathcal{J}_1+\mathcal{J}_2+\mathcal{J}_3+\mathcal{J}_4+\mathcal{J}_5.
    \end{split}
\end{align}

Now we estimate the right-hand side of \eqref{stimasulsecondolinearizzato}.
Using weighted Young's inequality we get 
\begin{align}\label{j1}
    \begin{split}
        \mathcal{J}_1\le &q \theta   \int_{B_{2R}} (\varepsilon+|\nabla u_\varepsilon|^2)^{kq} |u_{\varepsilon,ij}|^{2q-2}\psi ^2|\nabla u_{\varepsilon,ij}|^2\,dx 
        \\& + qC(p,k,\theta) \int_{B_{2R}} (\varepsilon+|\nabla u_\varepsilon|^2)^{kq-1} | u_{\varepsilon,ij}|^{2q-2}|D^2u_\varepsilon|^4\psi^2\,dx ,
    \end{split}
\end{align}

where $C(p,k,\theta)$ is a positive constant.

In the same way we estimate

\begin{align}\label{j2}
    \begin{split}
        \mathcal{J}_2\le &\theta \int_{B_{2R}} (\varepsilon+|\nabla u_\varepsilon|^2)^{kq} |u_{\varepsilon,ij}|^{2q-2}\psi ^2|\nabla u_{\varepsilon,ij}|^2\,dx 
        \\& + C(p,\theta) \int_{B_{2R}} (\varepsilon+|\nabla u_\varepsilon|^2)^{kq} | u_{\varepsilon,ij}|^{2q}|\nabla \psi|^2\,dx ,
    \end{split}
\end{align}
  where $C(p,\theta)$ is a positive constant.

  Obviously, we estimate $\mathcal{J}_3$ in the following way 
  \begin{equation}\label{j3}
      \mathcal{J}_3 \le  qC(p,k) \int_{B_{2R}} (\varepsilon+|\nabla u_\varepsilon|^2)^{kq-1} |u_{\varepsilon,ij}|^{2q-2}|D^2u_\varepsilon|^4\psi^2\,dx.
  \end{equation}

	For the last term $\mathcal{J}_4$, by standard Young's inequality we have 

 \begin{align}\label{j4}
    \begin{split}
        \mathcal{J}_4&\le  C(p)\int_{B_{2R}} (\varepsilon+|\nabla u_\varepsilon|^2)^{kq-1} | u_{\varepsilon,ij}|^{2q-2}|D^2u_\varepsilon|^4 \psi^2\,dx  
        \\& + C(p) \int_{B_{2R}} (\varepsilon+|\nabla u_\varepsilon|^2)^{kq} |u_{\varepsilon,ij}|^{2q}|\nabla \psi|^2\,dx ,
	\end{split}
 \end{align}

where $C(p)$ is a positive constant.

Using \eqref{j1}, \eqref{j2}, \eqref{j3} and \eqref{j4} in \eqref{stimasulsecondolinearizzato} we obtain 

\begin{align}\label{penultimasulsecondolinearizzato}
    \begin{split}
        &((2q-1)\min\{1,p-1\}-(q+1)\theta)\int_{B_{2R}} (\varepsilon+|\nabla u_\varepsilon|^2)^{kq} |u_{\varepsilon,ij}|^{2q-2}\psi ^2|\nabla u_{\varepsilon,ij}|^2\,dx
       \\&\le   qC(p,k,\theta)\int_{B_{2R}} (\varepsilon+|\nabla u_\varepsilon|^2)^{kq-1} | u_{\varepsilon,ij}|^{2q-2}|D^2u_\varepsilon|^4 \psi^2\,dx 
        \\& + C(p,\theta) \int_{B_{2R}} (\varepsilon+|\nabla u_\varepsilon|^2)^{kq} | u_{\varepsilon,ij}|^{2q}|\nabla \psi|^2\,dx \\&+ \int_{B_{2R}}|f_{ij}|(\varepsilon+|\nabla u_\varepsilon|^2)^{\frac{2kq+2-p}{2}}|u_{\varepsilon,ij}|^{2q-1}\psi^2\,dx.
    \end{split}
\end{align}
Therefore, by choosing $\theta:=\theta(p)$ sufficiently small, we infer that
\begin{equation*}
    \begin{split}
        &q^2\int_{B_{2R}} (\varepsilon+|\nabla u_\varepsilon|^2)^{kq} |u_{\varepsilon,ij}|^{2q-2}\psi ^2|\nabla u_{\varepsilon,ij}|^2\,dx
       \\&\le   q^2C(p,k)\int_{B_{2R}} (\varepsilon+|\nabla u_\varepsilon|^2)^{kq-1} | u_{\varepsilon,ij}|^{2q-2}|D^2u_\varepsilon|^4 \psi^2\,dx 
        \\& + qC(p) \int_{B_{2R}} (\varepsilon+|\nabla u_\varepsilon|^2)^{kq} | u_{\varepsilon,ij}|^{2q}|\nabla \psi|^2\,dx \\&+ qC(p)\int_{B_{2R}}|f_{ij}|(\varepsilon+|\nabla u_\varepsilon|^2)^{\frac{2kq+2-p}{2}}|u_{\varepsilon,ij}|^{2q-1}\psi^2\,dx.
    \end{split}
\end{equation*}
Concerning the second summand in the r.h.s. of \eqref{gradio} we note that it can be trivially estimated by the first term in the r.h.s. of the above inequality. This proves \eqref{stimalinearizzatosecondo} and concludes the proof.
\end{proof}

In the following proposition, in the spirit of \cite{DEG,MO,serrin}, we prove that, when $n\geq 3$, we may estimate the $L^{2^{*}q}$-norm of $g_{\varepsilon}$ defined in \eqref{g}, on a ball of radius $h'$ in terms of some power of the $L^{\nu(q-\frac1k)}$-norm, with $\nu<2^*$, on a larger ball $B_{h}$. This is the starting point for a Moser's iteration argument, which will be performed in Prop. \ref{prop:moser}.
\begin{prop}\label{dinodino}
   Let $\Omega$ be a domain in $\R^n$, $n\ge 3$, and $B_{2R}\subset\subset \Omega$. Let $u_\varepsilon$ be a weak solution of the regularized problem \eqref{eq:problregol}. Assume $f\in C^2(\Omega)$.
   For $0< k\leq  1$ fixed, there exists  $\mathfrak{C}:=\mathfrak{C}(k,l,n)>0$  such that if $|p-2|<\mathfrak{C}$ then the following holds true. For any  $q\ge \frac 1k$ and any $\nu$ such that $2l/(l-1)<\nu<2^{*}$, there exists   $\mathcal{C}>0$ depending on $k,n,l$, $R,p$, $\| \nabla u\|_{L^{\infty}(B_{2R})}$ and $\|f\|_{W^{2,l}(B_{2R})}$ such that
\begin{align}\label{cane477}
    \begin{split}
    \|g_{\varepsilon}\|_{L^{2^*q}(B_{h'})}^q\leq \mathcal{C}\frac{q}{h-h'}\|g_{\varepsilon}\|_{L^{\nu(q-\hat{s})}(B_h)}^{q-\hat{s}}, \quad \hat{s}:=\frac1k,
    \end{split}
\end{align}
for any $0<h'<h<R$  and where $g_{\varepsilon}$ is defined in \eqref{g}.

\end{prop}
\begin{proof}
Let $h$ and $h'$ be real numbers satisfying $h'<h\le R$. Let $\psi(x)$ be the standard cut-off function such that $\psi(x)=1$ in the ball $B_{h'}$, $\psi(x) =0$ in the complement of the ball $B_h$, and $|\nabla \psi|\le 2/(h-h')$. 
Recall the notation \eqref{g}, by taking the square root of the inequality \eqref{stimalinearizzatosecondo}, we obtain

\begin{align}\label{zeus}
    \begin{split}
    \|\nabla g_{\varepsilon}^{q}\|_{L^{2}(B_{h'})}
       \le  & C(p,k)\Bigg(q\left(\int_{B_{h}} (\varepsilon+|\nabla u_\varepsilon|^2)^{kq-1} | u_{\varepsilon,ij}|^{2q-2}|D^2u_\varepsilon|^4\psi^2\,dx \right)^{\frac{1}{2}}
       \\&+\sqrt q \left( \int_{B_{h}}|f_{ij}|(\varepsilon+|\nabla u_\varepsilon|^2)^{\frac{2kq+2-p}{2}}|u_{\varepsilon,ij}|^{2q-1}\psi^2\,dx\right)^{\frac 12}
        \\
         &+\sqrt q\|g_{\varepsilon}^q \nabla\psi\|_{L^2(B_h)}\Bigg).
    \end{split}
\end{align}
Summing the same quantity on both sides of the previous inequality \eqref{zeus}, we obtain
\begin{align}\label{zeusmorto}
    \begin{split}
    \|g_{\varepsilon}^{q}\|_{W^{1,2}(B_{h'})}\le    & C(p,k)\Bigg(q\left(\int_{B_{h}} (\varepsilon+|\nabla u_\varepsilon|^2)^{kq-1} | u_{\varepsilon,ij}|^{2q-2}|D^2u_\varepsilon|^4\psi^2\,dx \right)^{\frac{1}{2}}
       \\&+\sqrt q\left( \int_{B_{h}}|f_{ij}|(\varepsilon+|\nabla u_\varepsilon|^2)^{\frac{2kq+2-p}{2}}|u_{\varepsilon,ij}|^{2q-1}\psi^2\,dx\right)^{\frac 12}
        \\&+\sqrt q\|g_{\varepsilon}^q \nabla\psi\|_{L^2(B_h)}\Bigg)+\|g_{\varepsilon}^{q}\|_{L^{2}(B_{h'})}.\\
\end{split}
\end{align}
By using  Sobolev's inequality 
\begin{align}\label{cane}
    \begin{split}
\|g_{\varepsilon}^{q}\|_{L^{2^* }(B_{h'})}\le
    & C_SC(p,k)\Big[q\left(\int_{B_{h}} (\varepsilon+|\nabla u_\varepsilon|^2)^{kq-1} | u_{\varepsilon,ij}|^{2q-2}|D^2u_\varepsilon|^4\psi^2\,dx \right)^{\frac{1}{2}}
    \\&+\sqrt q\left( \int_{B_{h}}|f_{ij}|(\varepsilon+|\nabla u_\varepsilon|^2)^{\frac{2kq+2-p}{2}}|u_{\varepsilon,ij}|^{2q-1}\psi^2\,dx\right)^{\frac 12}
        \\&
    +\frac{\sqrt q}{h-h'}\|g_{\varepsilon}^q \|_{L^2(B_h)}\Big]
         :=\mathcal{I}_1+\mathcal{I}_2+\mathcal{I}_3,
    \end{split}
\end{align}
where $C_S$ denotes the Sobolev's constant and where we have used the inequalities  $h'\leq h$ and $|\nabla \psi|\le 2/(h-h')$.

We estimate $\mathcal{I}_1$. Fix $\hat s:=1/k\ge 1$, independent of $q$, to be chosen later, and $\frac{2l}{l-1}<\nu<2^{*}$. Note that  $\frac{2l}{l-1}<2^{*}$ since by hypothesis we have $l>\frac{n}{2}$.  By using Holder's inequality with exponents $(\frac{\nu}{2},\frac{\nu}{\nu-2})$ we obtain
\begin{align}\label{cne44}
\begin{split}
    \mathcal{I}_1 \le 
        & q C_SC(p,k) \left(\int_{B_{h}}  | u_{\varepsilon,ij}|^{(2\hat s-2)\frac{\nu}{\nu-2}}|D^2u_\varepsilon|^{4\frac{\nu}{\nu-2}}\,dx\right)^{\frac {\nu-2}{2\nu}}\times 
        \\& \times \left(\int_{B_{h}} (\varepsilon+|\nabla u_\varepsilon|^2)^{(kq-1)\frac{\nu}{2}} | u_{\varepsilon,ij}|^{\nu(q-\hat s)} \psi^{\nu}\,dx\right)^{\frac 1\nu}.
    \end{split}
\end{align}
By Theorem \ref{teoremacalderon}, see Remark \ref{regolarizziamo}, there exists a positive constant $\mathcal{C}$ independent of $\varepsilon$ and depending on $n,\hat s,\nu,R,p,\| \nabla u\|_{L^{\infty}(B_{2R})},\|f\|_{C^{0,\beta'}(B_{2R})},$ and $\|f\|_{W^{1,1}( B_{2R})}$  such that 

\begin{equation}\label{canestro}
   \left(\int_{B_{h}}  | u_{\varepsilon,ij}|^{(2\hat s-2)\frac{\nu}{\nu-2}}|D^2u_\varepsilon|^{4\frac{\nu}{\nu-2}}\,dx\right)^{\frac {\nu}{2(\nu-2)}}\le  \left(\int_{B_{h}}  |D^2 u_\varepsilon|^{\frac{2\nu}{\nu-2}(\hat s+1)}\,dx\right)^{\frac {\nu-2}{2\nu}}\le \mathcal{C}.
\end{equation}
By \eqref{canestro} and the inequality \eqref{cne44} we infer that 
\begin{equation}\label{cne43}
    \mathcal{I}_1\le\hat{ \mathcal{C}} q \left(\int_{B_{h}} (\varepsilon+|\nabla u_\varepsilon|^2)^{kq-1\frac{\nu}{2}} | u_{\varepsilon,ij}|^{\nu(q-\hat s)} \psi^\nu\,dx\right)^{\frac 1\nu}
\end{equation}
where $\hat{ \mathcal{C}}=\hat{\mathcal{C}}(k,n,\nu,R,p,\| \nabla u\|_{L^{\infty}(B_{2R})},\|f\|_{C^{0,\beta'}( B_{2R})},\|f\|_{W^{1,1}( B_{2R})})$ is a positive constant independent of $\varepsilon$ and $q$.

We estimate the term $\mathcal{I}_2$. To begin with, we notice that, by  \cite[Theorem 1.7]{L2} (see also \cite{DB} and \cite[Theorem $2.1$]{Cma2}), we have 
\begin{equation}\label{uniformitàdeigradienti}
    \|\nabla u_\varepsilon\|_{L^{\infty}(B_{2h})}\le C,
\end{equation}
where $C(p,n,R,\|f\|_{L^s(B_{2R})})$ is a positive constant not depending on $\varepsilon$ and $s>n$.

Since $p<4$ and since $\nabla u_{\varepsilon}$ is uniformly bounded in $L^{\infty}$ we infer that $\\(\varepsilon+|\nabla u_\varepsilon|^2)^{4-p}$ is uniformly bounded in $L^{\infty}$.
From the latter fact, by \eqref{uniformitàdeigradienti} and by using Holder's inequality with exponents $(l,l/(l-1))$ we obtain
\begin{equation}
\begin{split}
  \mathcal{I}_2\le &\sqrt qC_SC(p,k,R,\|f\|_{L^s(B_{2R})})
  \left( \int_{B_{h}}|f_{ij}|^l\,dx\right)^{\frac {1}{2l}}\times
  \\&\times\left( \int_{B_{h}}(\varepsilon+|\nabla u_\varepsilon|^2)^{(kq-1)\frac{l}{l-1}}|u_{\varepsilon,ij}|^{2(q-\hat s)\frac{l}{l-1}}|u_{\varepsilon,ij}|^{(2\hat s-1)\frac{l}{l-1}}\psi^{\frac{2l}{l-1}}\,dx\right)^{\frac {l-1}{2l}}.
    \end{split}
\end{equation}
Since $f\in W^{2,l}(B_{2R})$, by using Holder's inequality with exponents $\left(\frac{\nu(l-1)}{2l},\frac{\nu(l-1)}{l(\nu-2)-\nu}\right)$, we have 
\begin{equation}\label{prestimaI_2}
\begin{split}
  \mathcal{I}_2\le &\sqrt qC
  \left( \int_{B_{h}}|u_{\varepsilon,ij}|^{(2\hat s-1)\frac{l\nu}{l(\nu-2)-\nu}}\,dx\right)^{\frac {l(\nu-2)-\nu}{2l\nu}}\times
  \\&\times\left( \int_{B_{h}}(\varepsilon+|\nabla u_\varepsilon|^2)^{(kq-1)\frac{\nu}{2}}|u_{\varepsilon,ij}|^{(q-\hat s)\nu}\psi^{\nu}\,dx\right)^{\frac {1}{\nu}},
    \end{split}
\end{equation}
with $C=C(p,k,n,R,\|f\|_{L^s(B_{2R})},\|f\|_{W^{2,l}(B_{2R})})$.

By \eqref{prestimaI_2} and by Theorem \ref{teoremacalderon}, see Remark \ref{regolarizziamo},  we deduce that 
\begin{equation}\label{stimaI_2}
    \mathcal{I}_2\le\dot{ \mathcal{C}} \sqrt q \left(\int_{B_{h}} (\varepsilon+|\nabla u_\varepsilon|^2)^{(kq-1)\frac{\nu}{2}} | u_{\varepsilon,ij}|^{\nu(q-\hat s)} \psi^\nu\,dx\right)^{\frac 1\nu},
\end{equation}
where $\dot{ \mathcal{C}}=\dot{\mathcal{C}}(k,n,\nu,R,p,\| \nabla u\|_{L^{\infty}(B_{2R})},\|f\|_{W^{2,l}( B_{2R})})$ is a positive constant independent of $\varepsilon$ and $q$.

In order to estimate  the term $\mathcal{I}_3$, we can reason in a similar way as done for the estimate on $\mathcal{I}_1$. Exploiting the uniform boundedness in $\varepsilon$ of the $L^{\infty}$-norm of the gradient of $u_{\varepsilon}$ and  using Holder's inequality with exponents $(\frac{\nu}{2},\frac{\nu}{\nu-2})$, we obtain
\begin{align}\label{prestimaI_4}
\begin{split}
    \mathcal{I}_3 \le 
        & \frac{\sqrt q}{h-h'} C_SC(p,k,n,R,\|f\|_{L^s(B_{2R})}) \left(\int_{B_{h}}  | u_{\varepsilon,ij}|^{2\hat s\frac{\nu}{\nu-2}}\,dx\right)^{\frac {\nu-2}{2\nu}}\times 
        \\& \times \left(\int_{B_{h}} (\varepsilon+|\nabla u_\varepsilon|^2)^{(kq-1)\frac{\nu}{2}} | u_{\varepsilon,ij}|^{\nu(q-\hat s)} \psi^{\nu}\,dx\right)^{\frac 1\nu}.
    \end{split}
\end{align}
By Theorem \ref{teoremacalderon} (see also Remark \ref{regolarizziamo}) we get
\begin{align}\label{cne444}
\begin{split}
    \mathcal{I}_3 \le 
        & \frac{\sqrt q}{h-h'} \tilde{ \mathcal{C}}  \left(\int_{B_{h}} (\varepsilon+|\nabla u_\varepsilon|^2)^{(kq-1)\frac{\nu}{2}} | u_{\varepsilon,ij}|^{\nu(q-\hat s)} \psi^\nu\,dx\right)^{\frac 1\nu}\\
    \end{split}
\end{align}
where $\mathcal{\tilde C}=\tilde{\mathcal{C}}(k,n,\nu,R,p,\| \nabla u\|_{L^{\infty}(B_{2R})},\|f\|_{C^{0,\beta'}( B_{2R})})$ is again a positive constant independent of $q$ and $\varepsilon$.
We remark that in the estimates \eqref{canestro}, \eqref{prestimaI_2} and \eqref{prestimaI_4}   we have applied Theorem \ref{teoremacalderon} with $D^2u_\varepsilon\in L^{\hat q}(B_{2R})$ where 
\begin{equation}\label{cappuccio}
\hat{q}=\max\left\{(\hat s+1)\frac{2\nu}{\nu-2},(2\hat s-1)\frac{l\nu}{l(\nu-2)-\nu}\right\}.
\end{equation}
Collecting \eqref{cane}, \eqref{cne43}, \eqref{stimaI_2},  and \eqref{cne444}, using that $0\leq\psi\leq 1$  we infer that
\begin{align*}
    \begin{split}
    \|g_{\varepsilon}^q\|_{L^{2^*}(B_{h'})}\leq
    \mathcal{C}\frac{q}{h-h'}   \left(\int_{B_{h}} (\varepsilon+|\nabla u_\varepsilon|^2)^{(kq-1)\frac{\nu}{2}} | u_{\varepsilon,ij}|^{\nu(q-\hat s)} dx\right)^{\frac 1\nu},
    \end{split}
\end{align*}
where $\mathcal{C}=\mathcal{C}(k,n,\nu,l,R,p,\| \nabla u\|_{L^{\infty}(B_{2R})},\|f\|_{W^{2,l}( B_{2R})})$ is independent of $q$ and $\varepsilon$. We conclude by noticing that $\nu(kq-1)=\nu k(q-\hat s)$.\end{proof}
\begin{rem}\label{pluto}
    The case $n=2$ is obtained in a very similar way. The only difference is in the use of the Sobolev inequality to estimate from below the left hand side of \eqref{zeusmorto}. In this framework we can use the embedding $W^{1,2}_{loc}(\Omega)\hookrightarrow L^{r}_{loc}(\Omega)$ with $r>\frac{2l}{l-1}$. Then one has to continue the proof by choosing $2<\frac{2l}{l-1}<\nu<r$, and eventually obtain the \eqref{cane477} with $2^*\rightsquigarrow r$.
\end{rem}

Owing to Prop. \ref{dinodino} and Remark \ref{pluto}, we can perform a Moser's iteration argument and obtain that $(\varepsilon+|\nabla u_\varepsilon|^2)^{\frac{k}{2}}  u_{\varepsilon,ij}$ is locally (uniformly in $\varepsilon$) bounded in $L^{\infty}$.
\begin{prop}\label{prop:moser}
Fix $0<R'<R$.
Under the hypotheses of Prop. \ref{dinodino}, for any $k>0$ 
 there exists a constant $\mathcal{C}>0$ depending on $k,n,l,R'$, $R,p$, $\| \nabla u\|_{L^{\infty}(B_{2R})}$ and $\|f\|_{W^{2,l}( B_{2R})}$ such that
\begin{equation}\label{frtm1}
\left\|(\varepsilon+|\nabla u_\varepsilon|^2)^{\frac{k}{2}}  u_{\varepsilon,ij}\right\|_{L^{\infty}(B_{R'})}\le \mathcal{C}. 
\end{equation}
\end{prop}
\begin{proof}
Consider first the case $0<k<1$. Recall the function $g_{\varepsilon}$ defined in \eqref{g}, and the inequality \eqref{cane477}.
We use a Moser's type  iteration. For $i\in \N_0$ and $\hat{s}=1/k$, we set 
\begin{equation} \label{definizionediq_n}
			\begin{cases}
				q_0 =  \hat{s} \\
				q_{i+1} =\frac{2^*}{\nu}q_i+\hat s.\\
			\end{cases}
		\end{equation}
We note that with this definition we have for any $i\geq 0$
\begin{equation}\label{upperqi}
q_i=\hat{s}\sum_{j=0}^i\left(\frac{2^*}{\nu}\right)^j\leq \hat{s}(i+1)\left(\frac{2^*}{\nu}\right)^i.
\end{equation}
Using the choice $h=h_i=R'(1+2^{-i})$, $h'=h_{i+1}$ and $q=q_i$, the inequality \eqref{cane477} becomes
\begin{align}\label{cane478}
    \begin{split}
    \|g^{q_i}_{\varepsilon}\|_{L^{2^*}(B_{h_{i+1}})}\leq 2\mathcal{C}2^i{q}_i\|g_{\varepsilon}\|_{L^{\nu(q_i-\hat{s})}(B_{h_i})}^{q_i-\hat{s}}.
    \end{split}
\end{align}
with $\mathcal{C}=\mathcal{C}(k,n,l,R',R,p,\| \nabla u\|_{L^{\infty}(B_{2R})}, \|f\|_{W^{2,l}( B_{2R})})$ positive constant.
Set $p_{i+1}=2^*q_i$,  by \eqref{definizionediq_n} we infer that $p_i=\nu(q_i-\hat s)$, so that 

\begin{align}\label{canestro478}
    \begin{split}
    &\|g_{\varepsilon}\|_{L^{p_{i+1}}(B_{h_{i+1}})}
   \le   \left(2\mathcal{ C} 2^i q_i\right)^{\frac{2^*}{p_{i+1}}}
   \|g_{\varepsilon}\|_{L^{p_i}(B_{h_i})}^{\frac{2^*}{\nu}\frac{p_{i}}{p_{i+1}}}.
    \end{split}
\end{align}
The iteration yields

\begin{align}\label{camastra}
    \begin{split}
    &\left(\int_{B_{h_{i+1}}} (\varepsilon+|\nabla u_\varepsilon|^2)^{\frac{kp_{i+1}}{2}}  |u_{\varepsilon,ij}|^{p_{i+1}}\,dx\right)^{\frac{1}{p_{i+1}}}
   \\&\le  \prod_{k=0}^i [(2\mathcal{ C} 2^{k} q_k)^{\frac{(2^*)^{i-k}}{\nu^{i-k}}}]^{\frac{2^*}{p_{i+1}}} \left(\int_{B_{2R'}}  |g_{\varepsilon}|^{\nu( q_0-\hat s)} \,dx\right)^{\frac {(2^*)^{i+1}}{\nu^{i+1}p_{i+1}}}
   \\&= |B_{2R'}|^{\frac {(2^*)^{i+1}}{\nu^{i+1}p_{i+1}}}\prod_{k=0}^i [(2\mathcal{ C} 2^{k} q_k)^{\frac{(2^*)^{i-k}}{\nu^{i-k}}}]^{\frac{2^*}{p_{i+1}}},
    \end{split}
\end{align}
where in the last step we used that $q_0=\hat{s}$ and by $|B_{2R'}|$ the measure of $B_{2R'}$.
We prove that  the r.h.s. of the above inequality is bounded in $i$. Recalling that $p_{i+1}=2^{*}q_i$ and the trivial inequality $q_i\geq \hat{s}(2^*/\nu)^i$,  we have 
\begin{equation}\label{mimmo}\frac{1}{p_{i+1}}\leq \left(\frac{\nu}{2^*}\right)^i\frac{1}{\hat{s}2^*},\end{equation} so that    $|B_{2R'}|^{\frac {(2^*)^{i+1}}{\nu^{i+1}p_{i+1}}}$
is bounded from above by a constant independent of $i$. We analyze the second factor in the r.h.s. of \eqref{camastra}, i.e. 
\begin{equation}\label{camastra1}
    \prod_{k=0}^i [(2\mathcal{ C} 2^{k} q_k)^{\frac{(2^*)^{i-k}}{\nu^{i-k}}}]^{\frac{2^*}{p_{i+1}}}=\exp\left({\sum_{k=0}^i\left(\frac{2^*}{\nu}\right)^{i-k}\frac{2^*}{p_{i+1}}\log (2\mathcal{C}2^kq_k)}\right).
\end{equation}
Since $2\mathcal{C}2^kq_k\geq 1$, by using \eqref{mimmo} we obtain
\begin{equation*}
    \eqref{camastra1}\leq \exp\left(\sum_{k=0}^i\left(\frac{2^*}{\nu}\right)^{-k}\frac{1}{\hat{s}}\log(2\mathcal{C}2^kq_k)\right).
\end{equation*}
We use \eqref{upperqi} to continue the chain of inequality and infer that there exists $M>1$ such that
\begin{equation*}
    \eqref{camastra1}\leq \exp\left(\sum_{k=0}^i\left(\frac{\nu}{2^*}\right)^k\frac{1}{\hat{s}}k \log(M)\right),
\end{equation*}
which is bounded in $i$ since $\nu/2^*<1$.
Passing to the limit for $i\rightarrow \infty$ in \eqref{camastra}, and noticing that $\|(\varepsilon+|\nabla u_\varepsilon|^2)^{\frac{k}{2}}  u_{\varepsilon,ij}\|_{L^{\infty}(B_{R'})}\leq \lim_{i\rightarrow\infty}\|(\varepsilon+|\nabla u_\varepsilon|^2)^{\frac{k}{2}}  u_{\varepsilon,ij}\|_{L^{p_i}(B_{h_i})}$, we eventually  obtain 
\begin{equation*}
\left\|(\varepsilon+|\nabla u_\varepsilon|^2)^{\frac{k}{2}}  u_{\varepsilon,ij}\right\|_{L^{\infty}(B_{R'})}\le \mathcal{C},    
\end{equation*}
 where $\mathcal C$ is a positive constant not depending on $\varepsilon$. \\
 Now we consider the case $k> 1$. If $k$ is not an integer, then
\[
(\varepsilon + |\nabla u_\varepsilon|)^k |D^2 u_\varepsilon|
= (\varepsilon + |\nabla u_\varepsilon|)^{\lfloor k \rfloor} 
  (\varepsilon + |\nabla u_\varepsilon|)^{k - \lfloor k \rfloor} |D^2 u_\varepsilon|
\]
\[
\leq \| \varepsilon + |\nabla u_\varepsilon| \|_{L^\infty}^{\lfloor k \rfloor}
   \, \| (\varepsilon + |\nabla u_\varepsilon|)^{k - \lfloor k \rfloor} 
   |D^2 u_\varepsilon| \|_{L^\infty}.
\]

In the last line, the first term is bounded (uniformly in $\varepsilon$) thanks to the uniform bound on 
$\|\nabla u_\varepsilon\|_{L^\infty}$ and since $\lfloor k \rfloor \geq 1$, while the second term has just been proven to be bounded, as $k - \lfloor k \rfloor \in (0,1)$.

If $k > 1$ is an integer, then
\[
(\varepsilon + |\nabla u_\varepsilon|)^k |D^2 u_\varepsilon|
= (\varepsilon + |\nabla u_\varepsilon|)^{k-1} \, 
  (\varepsilon + |\nabla u_\varepsilon|) |D^2 u_\varepsilon|,
\]
and the same reasoning applies.
\end{proof}
\begin{rem}\label{pluto2}
    The same thesis of Prop. \ref{prop:moser} holds true in the case $n=2$, it is enough to change $2^*$ with $r$ according to the notation of Remark \ref{pluto}.
\end{rem}
We are in position to prove Theorem \ref{thm:main}.
\begin{proof}[Proof of Theorem \ref{thm:main}]
Let us begin with the case $n=3$. Fix $x_0\in\Omega$ and let $R>0$ such that $B_{2R}(x_0)\subset\subset\Omega $ and let $R'<R$.\\ We prove that $|\nabla u|^k u_{ij}\in L^{\infty}(B_{R'}(x_0))$.
Consider $u_{\varepsilon}$ solution to \eqref{eq:problregol}. By Proposition \ref{prop:moser} it is enough to prove that $(\varepsilon+|\nabla u_{\varepsilon}|^2)^{k/2}u_{\varepsilon,ij}$ converges a.e. to  $|\nabla u|^{k}u_{ij}$.
  Assume for the moment
\begin{equation}\label{reg_f}
    f \in C^2(\Omega). 
\end{equation}
Let us consider a compact set $K\subset \subset B_{R'}$. We recall (see \cite{Anto2,DB,L2}) that, 
\begin{equation}\label{evvai655}
{u}_\varepsilon\rightarrow  u \quad \text{in the norm }\|\cdot\|_{C^{1,\beta}(K)}.
\end{equation}\  
By Schauder estimates (see \cite{GT}) we have that $\|{u}_\varepsilon\|_{C^{2,\beta}(\tilde K)}\leq C$
for any compact set $\tilde K\subset\subset ( B_{R'} \setminus Z_u)$, with $C$  positive constant not depending on $\varepsilon$. So, by the latter inequality, using Ascoli-Arzelà theorem, up to subsequence, it follows that the limit in \eqref{evvai655} holds in $C^2(\tilde{K}).$ So, by \eqref{frtm1}, for $\varepsilon\rightarrow 0$, we deduce that
\begin{equation}\label{frtm2}
\left||\nabla u|^ku_{ij}|\right|_{L^{\infty}(B_{R'}\setminus Z_u)}\le \mathcal{C}.    
\end{equation}

Moreover by Theorem \ref{teoremacalderon} we have $u_i\in W^{1,\hat q}(B_{2R})$, with $\hat{q}>1$ in \eqref{cappuccio}, and therefore, by using Stampacchia's theorem (see \cite[Theorem $6.19$]{S}) we get that $ \nabla u_i=0$ a.e on the set $\{|\nabla u|=0\}$. From this fact, it follows that
\begin{equation}\label{frtm3}
\left\||\nabla u|^ku_{ij}|\right\|_{L^{\infty}(B_{R'})}\le \mathcal{C},
\end{equation}
with $\mathcal{C}=\mathcal{C}(k,n,l,\nu,R,R',p,\| \nabla u\|_{L^{\infty}(B_{2R})},\|f\|_{W^{2,l}( B_{2R})})$.

 The last step consists of removing the assumption \eqref{reg_f}. Let $f \in W^{2,l}_{loc}(\Omega) $. By standard density argument we can infer that there exists a sequence $\{{f}_m\} \subset C^\infty(\Omega)$ such that
\begin{equation}\label{conv_f}
    {f}_m \rightarrow {f} \quad \text{in } W^{2,l}_{loc}(\Omega).
\end{equation}

To proceed, we consider a sequence $\{{u}_m\}$ of weak solutions to the following system
\begin{equation}
\label{system_step3}
\begin{cases}
-\operatorname{ div}({|\nabla{ u_m}|}^{p-2}\nabla{ u_m})=   { f_m}& \mbox{in $B_{2R}$}\\
 u_m= u  & \mbox{on  $\partial B_{2R}$}.
\end{cases}
\end{equation}

\noindent By \eqref{frtm3} applied to $u_m$ we have

\begin{equation}\label{stimaimportante}
\left\|\nabla u_m|^ku_{m,ij}\right\|_{L^{\infty}(B_{R'})}\le \mathcal{C}, 
\end{equation}

\noindent with $\mathcal{C}=\mathcal{C}(k,n,l,\nu,R,R',p,\| \nabla u_m\|_{L^{\infty}(B_{2R})},\|f_m\|_{W^{2,l}( B_{2R})})$.

First, we recall that  (see \cite[Theorem 1.7]{L2}, \cite{DB} and \cite[Theorem $2.1$]{Cma2}), we have 
\begin{equation}\label{uniformitàdeigradientisum}
    \|\nabla u_m\|_{L^{\infty}(B_{2R'})}\le C,
\end{equation}
where $C(p,R,n,\|f_m\|_{L^s(B_{2R})})$ is a positive constant and $s>n$.

Now, consider a compact set $K\subset \subset B_{R'}$. We recall that, 
\begin{equation}\label{evvai6558}
{u}_m\rightarrow  u \quad \text{in the norm }\|\cdot\|_{C^{1,\beta}(K)}.
\end{equation}\ 
By  \cite[Theorem 2.1]{Cma} (see also \cite[Theorem 1.1]{MMS}), we deduce 
\begin{equation}\label{convergenzaW22}
    u_m\rightarrow u \qquad \text{in } W^{2,2}_{loc}(\tilde K),
\end{equation}
for any compact set $\tilde K\subset\subset (B_{R'}\setminus Z_u)$.
By \eqref{conv_f}, \eqref{uniformitàdeigradientisum}, \eqref{evvai6558} and \eqref{convergenzaW22}, passing to the limit in \eqref{stimaimportante} we deduce
\begin{equation}\label{frtm4.0}
\left\||\nabla u|^ku_{ij}\right\|_{L^{\infty}(B_{R'})}\le \mathcal{C}, 
\end{equation}
 with $\mathcal{C}=\mathcal{C}(k,n,l,\nu,R,R',p,\| \nabla u\|_{L^{\infty}(B_{2R})},\|f\|_{W^{2,l}( B_{2R})})$.\\
 The case $n=2$ follows by Remarks \ref{pluto} and \ref{pluto2}, reasoning as above. This concludes the proof.

\end{proof}

\section*{Acknowledgements}
We warmly thank Berardino Sciunzi for suggesting this problem.\\
  The authors are supported by PRIN PNRR P2022YFAJH \emph{Linear and Nonlinear PDEs: New directions and applications.}  Felice Iandoli has been partially supported by \emph{INdAM-GNAMPA Project Stable and unstable phenomena in propagation of Waves in dispersive media} E5324001950001. D. Vuono has been partially supported by \emph{INdAM-GNAMPA Project Regularity and qualitative aspects of nonlinear PDEs via variational and non-variational approaches} E5324001950001.

\textbf{Declarations.} Data sharing is not applicable to this article as no datasets were generated or analyzed during the current study.
Conflicts of interest: The authors have no conflicts of interest to declare.


\begin{thebibliography}{99}

\bibitem{leoni}{\sc G. Leoni, } \newblock  A first course in Sobolev spaces, Second edition. Graduate Studies in Mathe- matics, 181. American Mathematical Society, Providence, RI, 2017. 

 \bibitem{Anto1} {\sc C.A. Antonini, G. Ciraolo, A. Farina}.
    \newblock Interior regularity results for inhomogeneous anisotropic quasilinear equations.
    \newblock {\em Mathematische Annalen} 387 (3), 1745-1776, 2023.

\bibitem{Anto2} {\sc C.A. Antonini, A. Cianchi, G. Ciraolo, A. Farina, V Maz'ya.} 
\newblock Global second-order estimates in anisotropic elliptic problems.
{\em Proc. Lond. Math. Soc.} (3) 130, no. 3, Paper No. e70034, 60 pp, 2025.

\bibitem{antonuovo} {\sc C.A. Antonini, G. Ciraolo, F. Pagliarin}.
\newblock Second order regularity for degenerate $p$-Laplace type equations with log-concave weights.
\newblock {\em J. Lond. Math. Soc.} (2) 112 (2025), no. 3, Paper No. e70299.

\bibitem{AKM} {\sc B. Avelin, T. Kuusi, G. Mingione}.
\newblock Nonlinear Calder´on-Zygmund theory in the limiting
 case.
\newblock {\em Arch. Ration. Mech. Anal.}, vol. 227, no. 2, pp. 663-714, 2018.



	\bibitem{BaCiDiMa} {\sc A. Balci, A. Cianchi, L. Diening, V. Maz'ya}. 
  \newblock A pointwise differential inequality and second-order regularity for nonlinear elliptic systems.
  \newblock {\em Math. Ann.}, 383, no. 3-4, 1775--1824, 2022.
		
        
		\bibitem{beni} {\sc D. Baratta, B. Sciunzi, D. Vuono.} \newblock Third order estimates and the regularity of the stress field for solutions to $p$-Laplace equations.
\newblock {\em Accepted on Commun. Contemp. Math.}.


\bibitem{BMV} {\sc D. Baratta, L. Muglia, D. Vuono.} \newblock Second order regularity for solutions to anisotropic degenerate elliptic equations.
\newblock  {\em J. Differential Equations} 435, Paper No. 113250, 29 pp, 2025.

\bibitem{CRS} {\sc D. Castorina, G. Riey, B. Sciunzi}.
    \newblock Hopf Lemma and regularity results for quasilinear anisotropic elliptic equations.
    \newblock {\em Calc. Var. Partial Differential Equations} 58, no. 3, Paper No. 95, 18 pp, 2019.


\bibitem{cellina} {\sc A. Cellina} 
\newblock The regularity of solutions to some variational problems, including the $p$-Laplace equation for $2 \le p < 3$.
		\newblock {\em ESAIM: COCV} 23, 1543--1553, 2017.

\bibitem{3} {\sc  A. Cianchi, V.G. Maz’ya.}
\newblock Global Lipschitz regularity for a class of quasilinear elliptic equa- tions. 
		\newblock {\em Comm. Partial Differential Equations}, 36(1),100--133, 2011.
		
		\bibitem{2} {\sc A. Cianchi, V.G. Maz’ya.}
  \newblock Global boundedness of the gradient for a class of nonlinear elliptic systems.
		\newblock {\em Arch. Ration. Mech. Anal.}, 212(1), 129--177, 2014.
		
		
		
		\bibitem{4} {\sc A. Cianchi, V.G Maz’ya.} \newblock Gradient regularity via rearrangements for p-Laplacian type elliptic boundary value problems.
		\newblock {\em J. Eur. Math. Soc. (JEMS)}, 16(3),571--595, 2014.

\bibitem{245} {\sc A. Cianchi, V.G. Maz’ya.} \newblock Global gradient estimates in elliptic problems under minimal data and domain regularity.
\newblock {\em Commun. Pure Appl. Anal.} 14, no. 1, 285-311, 2015.
  

\bibitem{Cma} {\sc A. Cianchi, V.G. Maz'ya.} 
  \newblock Second-order two-sided estimates in nonlinear elliptic problems. 
  \newblock {\em Archive for Rational Mechanics and Analysis}, 229, 569--599, 2018.
		

\bibitem{Cma2} {\sc A. Cianchi, V. Maz’ya}. {Global boundedness of the gradient for a class of nonlinear elliptic systems,} {\emph{Arch. Ration. Mech.
Anal.} 212, 129–177}, 2014.

\bibitem{Cmavett}{\sc A. Cianchi, V. Maz'ya}. 
  \newblock Optimal second-order regularity for the p-Laplace system. 
  \newblock {\em J. Math. Pures Appl. (9)}, 132 , 41--78, 2019.


\bibitem{CCR}{\sc  G. Ciraolo, A. Figalli, A. Roncoroni}. 
  \newblock  Symmetry results for critical anisotropic $p$-
Laplacian equations in convex cones. 
  \newblock {\em  Geom. Funct. Anal.}, vol. 30, no. 3, 2020.



\bibitem{DS} {\sc L. Damascelli, B. Sciunzi}.
    \newblock Regularity, monotonicity and symmetry of positive solutions of $m$-Laplace equations.
    \newblock {\em J. of Differential Equations}, 206, no.2, pp. 483--515, 2004.

\bibitem{FFM} {\sc C. De Filippis, F. De Filippis, M, Piccinini}.
\newblock Bounded minimizers of double phase problems at nearly linear growth.
\newblock {\em Preprint}
arXiv:2411.14325

		\bibitem{DM} {\sc C.  De Filippis, G. Mingione.} 
  \newblock On the regularity of minima of non-autonomous functionals. 
  \newblock {\em The Journal of Geometric Analysis}, 30, 1584--1626, 2020.
		
		
		\bibitem{DM2} {\sc C.  De Filippis, G. Mingione.} 
  \newblock Regularity for Double Phase Problems at
			Nearly Linear Growth. 
   \newblock {\em Arch. Rational Mech. Anal.},  247, no.5, Paper No. 85, 50 pp, 2023.
		
		\bibitem{DM3} {\sc C.  De Filippis, G. Mingione.}
  \newblock Nonuniformly elliptic Schauder theory. \newblock {\em Invent. Math.},  234, no.3, 1109--1196, 2023.

\bibitem {DEG} {\sc E. De Giorgi}.
\newblock Convergence problems for functionals and operators.
\newblock {\em Proc. Int. Meeting on Recent Methods in Nonlinear Analysis} (Rome,
1978), Pitagora, Bologna, pp. 131-188, 1979.

        
		\bibitem{DB} {\sc E. Di Benedetto}.
		\newblock $C^{1+\alpha}$ local regularity of weak solutions of degenerate elliptic equations. 
		\newblock {\em Nonlinear Anal.} 7(8), pp. 827--850, 1983.
		
		
		\bibitem{DiKaSc}{\sc L. Diening, P. Kaplick\'{y}, S. Schwarzacher}.
		\newblock{B{MO} estimates for the {$p$}-{L}aplacian}.
		\newblock {\em Nonlinear Anal.}, {75}, no. 2,  {637--650}, 2012.

        
		\bibitem{Dong} {\sc H. Dong, F. Peng, Y. Zhang, Y. Zhou}.
        \newblock Hessian estimates for equations involving p-Laplacian via a 
fundamental inequality
\newblock {\em Adv. Math}. 370, 107212, 2020.


\bibitem{12} {\sc F. Duzaar, G. Mingione.} 
\newblock Gradient continuity estimates. 
		\newblock {\em Calc. Var. Partial Differential Equations,} 39(3-4), 379--418, 2010.
		
		\bibitem{13} {\sc F. Duzaar, G. Mingione.} \newblock Gradient estimates via non-linear potentials. 
		\newblock {\em Amer. J. Math.}, 133(4), 1093--1149, 2011.

        
		\bibitem{GT} D.~Gilbarg, N.~S.~Trudinger. \emph{Elliptic Partial Differential Equations of Second Order}.
		\newblock {Springer, Reprint of the 1998 Edition}.


\bibitem{HS}
{\sc A. Haarala, S. Sarsa}. 
\newblock Global second order Sobolev-regularity of $p$-harmonic functions.
\newblock {\em J.
 Differ. Equations}, vol. 339, pp. 232-269, 2022.

\bibitem{IM}
{\sc T. Iwaniec, J. J. Manfredi}. 
\newblock Regularity of p-harmonic functions on the plane.
\newblock {\em Rev. Mat.
 Iberoam.}, vol. 5, no. 1-2, pp. 1-19, 1989.



\bibitem{KuuMin}
{\sc T. Kuusi, G. Mingione}. 
\newblock A nonlinear Stein theorem.
\newblock {\em Calc. Var. Partial Differential Equations}, 51, no. 1-2, 45--86, 2014.


\bibitem{16} {\sc T. Kuusi,  G. Mingione.} 
\newblock Vectorial nonlinear potential theory. 
		\newblock {\em J. Eur. Math. Soc. (JEMS)} 20(4), 929--1004, 2018.


\bibitem{LU} {\sc O.A. Ladyzhenskaya, N.N. Uralt’seva}. 
\newblock Linear and quasilinear equations. 
		\newblock {\em Math. Sci. Eng.},
 vol. 64, 1968.



        
		\bibitem{S} {\sc E. H. Lieb, M. Loss}.
\newblock Analysis, volume 14 of Graduate Studies in Mathematics.
\newblock {\em American Mathematical Society}, Providence, RI, 1997.



        \bibitem{LIB} {\sc G.M. Lieberman}. Boundary regularity for solutions of degenerate elliptic equations. Nonlinear Anal., 12(11), 1203–1219, 1988.


\bibitem{L2} {\sc G. M. Lieberman}.
\newblock The natural generalization of the natural conditions of Ladyzhenskaya and Ural'tseva for elliptic equations.
\newblock {\em Comm. Partial Differential Equations}, 16, no. 2-3, pp. 311-361, 1991.

\bibitem{Lou} {\sc Lou H.} \newblock On singular sets of local solutions to $p$-Laplacian equations.
		\newblock {\em Chinese Annals of Mathematics, Series B}, 29(5), 521--530, 2008.
		


\bibitem{MRS} {\sc C. Mercuri, G. Riey, and B. Sciunzi}.
		\newblock A regularity result for the $p$-Laplacian near uniform ellipticity. \newblock {\em SIAM J. Math. Anal.} 48:3, 2059--2075, 2016.

\bibitem{Minchievic} {\sc M.Miśkiewicz}.
  \newblock Fractional differentiability for solutions of the inhomogeneous $p$-Laplace system. 
  \newblock {\em Proc.  Amer. Math. Soc.}, 146, no.7, 3009--3017, 2018.

\bibitem{MMS} {\sc L. Montoro, L. Muglia, B. Sciunzi}.
\newblock Optimal second order boundary regularity for solutions to $p$-Laplace equations.
\newblock {\em Calculus of Variations and Partial Differential Equations}, 64(2), 50, 2025.

\bibitem{MMSV} {\sc L Montoro, L Muglia, B Sciunzi, D Vuono.} \newblock Regularity and symmetry results for the vectorial $p$-Laplacian.
		\newblock {\em Nonlinear Analysis, Theory, Methods and Applications}, 251, 113700, 2025.


	\bibitem{MO} {\sc J.K. Moser}
\newblock  On Harnack's theorem for elliptic differential elliptic equations.
\newblock  {\em Comm. on Pure and Applied Math.}, 14, pp. 577--591, 1961.

\bibitem{S1} {\sc B. Sciunzi}. 
\newblock Some results on the qualitative properties of positive solutions of quasilinear elliptic equations.
		\newblock {\em NoDEA. Nonlinear Differential Equations and Applications}, 14(3-4), 315–334, 2007.
		
		\bibitem{S2} {\sc B. Sciunzi.} \newblock Regularity and comparison principles for p-Laplace equations with vanishing source term.
		\newblock {\em Comm. Cont. Math.}, 16(6), 1450013, 20, 2014.
		
		

\bibitem {serrin} {\sc J. Serrin}. 
\newblock Local behaviour of solutions of quasilinear equations.
\newblock {\em Acta Math.} 111, pp. 247–302, 1964.
  
		\bibitem{T} {\sc P. Tolksdorf}.
		\newblock {\em Regularity for a more general class of quasilinear elliptic equations}.
		\newblock J. Differential Equations.  51(1), pp. 126--150, 1984.
		
		
	\end{thebibliography}
\end{document}